\documentclass[12pt]{article}

\usepackage{fullpage}
\usepackage{amsmath, amsthm, amsfonts, amssymb, amstext, mathrsfs, enumerate}
\usepackage{graphicx, ragged2e, lscape, framed, xcolor}
\usepackage{subfiles}

\theoremstyle{plain}
\newtheorem{theorem}{Theorem}[section]

\newtheorem{conjecture}[theorem]{Conjecture}

\newtheorem{claim}{Claim}[section]
\newtheorem{remark}[theorem]{Remark}

\numberwithin{equation}{section}
\allowdisplaybreaks

\newcommand{\affl}[3]{\noindent #1, Email: {\tt #2}\\ \textsc{#3}\\[1.5pt]}

\usepackage[pagebackref]{hyperref}
\hypersetup{
	colorlinks=true,
    urlcolor=purple,
	linkcolor=purple,
    citecolor=purple,
}

\DeclareMathOperator{\SRG}{SRG}
\DeclareMathOperator{\diag}{diag}

\DeclareMathOperator{\supp}{supp}

\def\1{\mbox{\boldmath $1$}}
\def\e{\mbox{\boldmath $e$}}
\def\x{\mbox{\boldmath $x$}}
\def\y{\mbox{\boldmath $y$}}
\def\z{\mbox{\boldmath $z$}}

\title{\textbf{Localization of the Caro--Wei bound and \\its applications to bipartiteness}}
\author{Aida Abiad \and Hitesh Kumar \and Shivaramakrishna Pragada}
\date{}

\begin{document}
\maketitle
\begin{abstract}
We confirm a conjecture of Brause, Randerath, Rautenbach and Schiermeyer (2016) by proving a localized lower bound on the independence number of a graph that strengthens the classical bounds of Fajtlowicz (1978) and of
Caro (1979) and Wei (1981), which in turn settles a conjecture by Bertram and Hor\'{a}k (1996). Our proof is based on a new Motzkin--Straus-type inequality involving local clique numbers and the independence number. We then apply the developed methods to study spectral and algebraic measures of graph bipartiteness. In particular, we extend a theorem of Brandt (1998) on spectral bipartiteness from regular $K_{r+1}$-free graphs to all $K_{r+1}$-free graphs, we improve a general upper bound for the least signless Laplacian eigenvalue of $K_{r+1}$-free graphs, and we disprove a conjecture of de Lima, Nikiforov and Oliveira (2016) in the case of $K_4$-free graphs.\\

\noindent
\textbf{Keywords:} Independence number, Local clique number, Motzkin--Straus inequality,  Signless Laplacian, Bipartiteness

\noindent
\textbf{MSC2020:} 05C50
\end{abstract}

\section{Introduction}


Let $G = (V, E)$ be a finite simple graph of order $n$, size $m$, maximum degree $\Delta(G)$, clique number $\omega(G)$ and independence number $\alpha(G)$. Let $\deg_G(v)$ denote the degree of a vertex $v\in V(G)$. For a vertex $v\in V(G)$, let $c(v)$ denote the size of the largest clique in $G$ containing $v$, that is, the largest clique in the closed of neighbourhood of $v$. Let $A(G)$ be the \emph{adjacency matrix} of the graph $G$. Since $A(G)$ is a real symmetric matrix, all eigenvalues of $A(G)$ are real and we denote them as   
\[ \lambda_1(G) \ge \lambda_2(G) \ge \cdots \ge \lambda_n(G).\]

Let $Q(G) = D(G) + A(G) $ be the \emph{signless Laplacian matrix} of $G$, where $D(G)$ is the diagonal matrix of the degrees of vertices in $G$. We denote the eigenvalues of $Q(G)$ by 
\[ q_1(G) \ge q_2(G) \ge \cdots \ge q_n(G)\ge 0.\]

\subsection{Localization of the Caro--Wei bound}

The following inequality is an immediate consequence of Brooks' Theorem, which states that for any graph $G$, it holds
\begin{equation}\label{eq:n_Delta}
\alpha(G)\ge \frac{n}{\Delta(G)+1}.    
\end{equation}

Fajtlowicz \cite{Fajtlowicz_1978,Fajtlowicz_1984} strengthened \eqref{eq:n_Delta} as follows.

\begin{theorem}[Fajtlowicz bound \cite{Fajtlowicz_1978,Fajtlowicz_1984}]\label{thm:Fajtlowicz}
For any graph $G$,
\begin{equation}\label{eq:Fajtlowicz}
    \alpha(G)\ge \frac{2n}{\Delta(G) + \omega(G)+1}.
\end{equation}
\end{theorem}

Caro \cite{Caro_1979} and Wei \cite{Wei_1981} localized \eqref{eq:n_Delta} using the degree sequence. 

\begin{theorem}[Caro--Wei bound \cite{Caro_1979, Wei_1981}]
    For any graph $G$, 
    \begin{equation}\label{eq:Caro_Wei}
      \alpha(G) \geq \sum_{v \in V(G)} \frac{1}{\deg(v)+1}.  
    \end{equation}
\end{theorem}

Note that \eqref{eq:Caro_Wei} and \eqref{eq:Fajtlowicz} are not comparable. In 1996, Bertram and Hor\'{a}k \cite{Bertram_Horak_1996} proposed the following conjecture as a strengthening of Fajtlowicz's bound \eqref{eq:Fajtlowicz}.

\begin{conjecture}[Bertram--Hor\'{a}k \cite{Bertram_Horak_1996}]\label{conj:Bertram_Horak}
    For any graph $G$, 
    \begin{equation}
         \alpha(G) \geq \sum_{v \in V(G)} \frac{2}{\deg(v)+\omega(G)+1}. 
    \end{equation}
\end{conjecture}

In recent years, there has been growing interest in \emph{vertex} or \emph{edge localizing} classical inequalities in graph theory. The localization process usually involves taking a classical inequality and  replacing some of its global parameters with neighborhood-dependent versions of the same parameters. For instance, in the case of the clique number, one usually replaces the global clique number $\omega(G)$ by the neighborhood dependent $c(v)$. Such localizations generally give stronger results as they capture local structure. We refer the reader to \cite{Aragao_Souza_2024, Balogh_Bradac_Lidicky_2025, Bradac_2022, Malec_Tompkins_2023} for localizations of inequalities in extremal graph theory. Recently, several spectral Tur\'{a}n type results have also been localized; see \cite{Kannan_Kumar_Pragada_2025,Kannan_Kumar_Pragada_2026, Liu_Sun_Wang_Wu_2026, Liu_Ning_2025_weighted, Liu_Ning_2026}.

In 2016, Brause, Randerath, Rautenbach, and Schiermeyer \cite{Brause_Randerath_Rautenbach_Schiermeyer_2016} proposed the following localized conjecture, which implies the Caro--Wei bound, the Fajtlowicz bound, and also Conjecture \ref{conj:Bertram_Horak} by Bertram--Hor\'{a}k.  

\begin{conjecture}[Brause--Randerath--Rautenbach--Schiermeyer \cite{Brause_Randerath_Rautenbach_Schiermeyer_2016}]\label{conj:Caro_Wei_localized}
    For any graph $G$, 
    \begin{equation}
         \alpha(G) \geq \sum_{v \in V(G)} \frac{2}{\deg(v)+c(v)+1}. 
    \end{equation}
\end{conjecture}

We note here that another different localization of Fajtlowicz's bound was established by Henning, L\"{o}wenstein, Southey, and Yeo \cite{Henning_Lowenstein_Southey_Yeo_2014}. Conjecture \ref{conj:Caro_Wei_localized} is known to be true for subquartic graphs, and for perfect graphs \cite{Brause_Randerath_Rautenbach_Schiermeyer_2016}. The triangle-free case follows from Shearer’s stronger degree-sequence bounds for triangle-free graphs \cite{Shearer_1983,Shearer_1991}. A further generalization of Conjecture \ref{conj:Caro_Wei_localized} has been noted by Kelly and Postle in \cite{Kelly_Postle_2024} as the stronger fractional local Reed conjecture through a local demands based coloring, which in turn is equivalent to a weighted formulation of Conjecture \ref{conj:Caro_Wei_localized}, see \cite{Kelly_Postle_2024} for more details. See also other local fractional Reed conjecture variations in \cite{Chudnovsky_King_Plumettaz_Seymour_2013,Edwards_King_2014}. 

Our first result is a proof of Conjecture \ref{conj:Caro_Wei_localized}, which is given in Section \ref{section:Caro_Wei_localization}. Our proof uses a Motzkin--Straus type inequality for the independence number (Theorem \ref{thm:localized_MS_independence}), showed in Section \ref{section:localized_MS_independence}, that we believe is of independent interest.

Although the first part of the paper is motivated by the above problems on the independence number, the techniques we will develop to tackle Conjecture \ref{conj:Caro_Wei_localized} have broader consequences. In particular, combining the localized independence number bound with a graph blow-up method yields a new extremal inequality, analogous in spirit to Tur\'an's theorem, which we will use to study the spectral and algebraic bipartiteness of $K_{r+1}$-free graphs.

\subsection{Spectral and algebraic bipartiteness in $K_{r+1}$-free graphs}

For a graph $G$, let $\epsilon_b(G)$ denote the minimum number of edges that need to be removed to make $G$ bipartite. The estimation of $\epsilon_b$ for $K_{r+1}$-free graphs is equivalent to the well-known Max-Cut problem for $K_{r+1}$-free graphs, and it has received a lot of attention in the last decades. One of the major open problems in this area is the following conjecture. 

\begin{conjecture}[Erd\H{o}s \cite{Erdos_1976, Erdos_Faudree_Pach_Spencer_1988}, Sudakov \cite{Sudakov_2007}]\label{conj:bipartiteness} Let $G$ be a $K_{r+1}$-free graph of order $n$. Then 
\[\frac{\epsilon_b(G)}{n^2} \le 
\begin{cases}
\frac{1}{25} & r= 2;\\
\frac{1}{9} & r = 3;\\
\frac{r-2}{4r} & r\ge 4 \text{ and even};\\
\frac{(r-1)^2}{4r^2} & r\ge 4 \text{ and odd}.
\end{cases}\]
\end{conjecture}

The special case $r=2$ was conjectured by Erd\H{o}s \cite{Erdos_1976, Erdos_Faudree_Pach_Spencer_1988} and remains open. The case $r=3$ was conjectured by Erd\H{o}s \cite{Erdos_Faudree_Pach_Spencer_1988} and resolved by Sudakov \cite{Sudakov_2007}. For $r\ge 4$, the conjecture was made by Sudakov \cite{Sudakov_2007}. The case $r=5$ was recently resolved by Hu, Lidick\'{y}, Martins, Norin and Volec \cite{Hu_Lidicky_Martins_Norin_Volec_2021}. 

Brandt \cite{Brandt_1998_local} related $\epsilon_b(G)$ with the adjacency eigenvalues of $G$ when $G$ is regular. 

\begin{theorem}[Brandt \cite{Brandt_1998_local}]\label{thm:spectral_bipartiteness}
Let $G$ be a regular graph of order $n$. Then
\begin{align*}
 \lambda_1(G)+ \lambda_n(G) &\leq \frac{4}{n}\epsilon_b(G).
\end{align*}
\end{theorem}

The quantity $\lambda_1(G) + \lambda_n(G)$ is called the \emph{spectral bipartiteness} of $G$. Note that Theorem \ref{thm:spectral_bipartiteness} is false in general; for instance, consider the graph $G = K_3 \cup 2K_1$, then $n = 5$, $\epsilon_b(G) = 1$ and \[\lambda_1(G) +\lambda_n(G) = 2-1 =1 > \frac{4\epsilon_b(G)}{n} = \frac{4}{5}.\]

Desai and Rao \cite{Desai_Rao_1994} showed that a better spectral parameter to consider is the least eigenvalue of the signless Laplacian. In fact, they show a Cheeger-type isoperimetric inequality for the least positive signless Laplacian eigenvalue. The following theorem follows from the general result by Desai and Rao \cite{Desai_Rao_1994}, and it is also noted explicitly by de Lima, Oliveira, Abreu and Nikiforov \cite{deLima_Oliveira_Abreu_Nikiforov_2011}, and by Fallat and Fan \cite{Fallat_Fan_2012}.

\begin{theorem}[Desai--Rao \cite{Desai_Rao_1994}]\label{thm:algebraic_bipartiteness}
Let $G$ be a graph on $n$ vertices. Then
\[q_n(G) \leq \frac{4}{n}\epsilon_b(G). \]
\end{theorem}

Note that $\lambda_1 + \lambda_n$ and $q_n$ of a graph are, in general, incomparable, but for regular graphs they are equal. Therefore, Theorem \ref{thm:algebraic_bipartiteness} strengthens Theorem \ref{thm:spectral_bipartiteness}. The quantity $q_n(G)$ is called the \emph{algebraic bipartiteness} of $G$. For a more systematic study of algebraic bipartiteness and its connections to graph bipartiteness, we refer the reader to the papers \cite{deLima_Oliveira_Abreu_Nikiforov_2011, Desai_Rao_1994,Fallat_Fan_2012}. 

In light of Conjecture \ref{conj:bipartiteness}, one expects analogous bounds to hold for spectral and algebraic bipartiteness of $K_{r+1}$-free graphs. Naturally inspired by the Erd\H{o}s conjecture for $\epsilon_b$ of triangle-free graphs, the triangle-free case has received significant attention. Indeed, Brandt \cite{Brandt_1998_local} conjectured that 
\begin{equation}
\lambda_1(G)+ \lambda_n(G) \leq \frac{4}{25}n    
\end{equation}
for regular triangle-free graphs $G$, and proved the weaker bound 
\begin{equation}\label{eq:Brandt_weak}
    \lambda_1(G)+ \lambda_n(G) \leq (3-2\sqrt{2})n.
\end{equation}
Csikv\'ari \cite{Csikvari_2022} extended \eqref{eq:Brandt_weak} to all triangle-free graphs. Abiad, Taranchuk, and Veluw \cite{Abiad_Taranchuk_Veluw_2026} extended Csikv\'ari's method and obtained upper bounds on spectral bipartiteness for graphs with high odd girth. Following this, Yip \cite{Yip_2026} using Chebyshev polynomials improved \eqref{eq:Brandt_weak} for triangle-free graphs further by a small amount and improved the upper bound even further for graphs with high odd girth.  

de Lima, Nikiforov and Oliveira \cite{deLima_Nikiforov_Oliveira_2016} proved that $q_n(G) \leq \frac{2}{9}n$ for all triangle free graphs. Motivated by the above progress, Balogh, Clemen, Lidick\'y, Norin, and Volec \cite{Balogh_Clemen_Felix_Lidicky_Norin_Volec_2023} used the technique of flag algebras to settle the triangle-free case for algebraic bipartiteness. 

\begin{theorem}[Balogh--Clemen--Lidick\'y--Norin--Volec \cite{Balogh_Clemen_Felix_Lidicky_Norin_Volec_2023}]\label{thm:balogh_qn_trianglefree}
Let $G$ be a triangle-free graph of order $n$. Then
\[ q_n(G) \leq \frac{15}{94}n.\]
\end{theorem}

In the regular case, $\lambda_1 + \lambda_n = q_n$, and so Theorem \ref{thm:balogh_qn_trianglefree} also settles the conjecture of Brandt for regular triangle-free graphs. 

For general $K_{r+1}$-free graphs, the following results are known.

\begin{theorem}\label{thm:spectral_bipartiteness_r_old} Let $G$ be a $K_{r+1}$-free graph of order $n$ where $r\ge 3$. 
\begin{enumerate}[$(i)$]
\item \emph{(Brandt \cite{Brandt_1998_local}).} If $G$ is regular, then 
    \[\frac{\lambda_1(G) + \lambda_n(G)}{n} \leq \left(2\sqrt{1-\frac{1}{r}}-1\right)^2 .\]
\item \emph{(de Lima--Nikiforov--Oliveira \cite{deLima_Nikiforov_Oliveira_2016})}. We have 
    \[ \frac{q_n(G)}{n} \leq  1 - \frac{3}{3r-1}.\]
\end{enumerate}
\end{theorem}

\begin{remark}
    When $r = 3,5$, the results of Sudakov \cite{Sudakov_2007} and Hu et al. \cite{Hu_Lidicky_Martins_Norin_Volec_2021} for $\epsilon_b$ give best-known bounds for $q_n$ by Theorem \ref{thm:algebraic_bipartiteness}, and $\lambda_1 + \lambda_n$ in the regular case by Theorem \ref{thm:spectral_bipartiteness}.
\end{remark}

Our second main result is an extension of the bound of Brandt for spectral bipartiteness in Theorem \ref{thm:spectral_bipartiteness_r_old} $(i)$ to all graphs.  

\begin{theorem}\label{thm:Brandt_generalization} Let $G$ be a $K_{r+1}$-free graph of order $n$. Then
    \[\frac{\lambda_1(G) + \lambda_n(G)}{n} \leq \left(2\sqrt{1-\frac{1}{r}}-1\right)^2.\]
\end{theorem}

We prove Theorem \ref{thm:Brandt_generalization} in Section \ref{section:spectral_bipartiteness}. Our proof uses an inequality (see Theorem \ref{thm:clique_independence_edge}) proved in Section \ref{section:graph_blowup_trick}, analogous to Tur\'{a}n's inequality, that we believe is of independent interest.

Furthermore, we improve the upper bound on algebraic bipartiteness of de Lima et al. from Theorem \ref{thm:spectral_bipartiteness_r_old} $(ii)$ as follows. 

\begin{theorem}\label{thm:q_min_new_bound} Let $G$ be a $K_{r+1}$-free graph of order $n$. Then
    \[ \frac{q_n(G)}{n} \le \left(1-\frac{1}{r}\right)^2.\]
\end{theorem}

We prove Theorem \ref{thm:q_min_new_bound} in Section \ref{section:algebraic_bipartiteness}. Our proof is a consequence of a recent result of Liu, Tang, and Zhang \cite{Liu_Tang_Zhang_2026}, which generalizes Motzkin--Straus Theorem to doubly non-negative matrices.  

de Lima, Nikiforov, and Oliveira \cite{deLima_Nikiforov_Oliveira_2016} conjectured the following upper bound for algebraic bipartiteness, which, if true, would be sharp for Tur\'{a}n graphs $T_r(n)$.   

\begin{conjecture}[de Lima, Nikiforov, and Oliveira \cite{deLima_Nikiforov_Oliveira_2016}]\label{conj:algebraic_K_r_free}  Let $G$ be a $K_{r+1}$-free graph of order $n$ where $r\ge 3$. Then
    \[ \frac{q_n(G)}{n} \le \left(1-\frac{2}{r}\right).\]
\end{conjecture}

A stronger conjecture of de Lima et al. \cite{deLima_Nikiforov_Oliveira_2016} says that $q_n$ is maximized uniquely by the Tur\'{a}n graph $T_r(n)$ among $K_{r+1}$-free graphs of order $n$. However, Oboudi \cite{Oboudi_2022} disproved this stronger conjecture by showing that there are other complete multipartite graphs without balanced partite sets that still achieve the same conjectured upper bound. We disprove Conjecture \ref{conj:algebraic_K_r_free} when $r = 3$.

\begin{theorem}\label{thm:K4_free_counterexample}
There is a sequence of $K_4$-free graphs $G_n$ of order $n$ such that
\[
 \lim_{n\to\infty}\frac{q_{n}(G_n)}{n}
 =\frac{60-4\sqrt5}{137+7\sqrt5}
 =0.3344572553\ldots>\frac{1}{3}.
\]  
\end{theorem}

We prove Theorem \ref{thm:K4_free_counterexample} in Section \ref{section:algebraic_bipartiteness}. It does not seem easy to disprove Conjecture \ref{conj:algebraic_K_r_free} for $r\ge 4$; see the discussion in Section \ref{section:open_problems}.

\section{A Motzkin-Straus type inequality for the independence number}
\label{section:localized_MS_independence}
Consider a simple graph $G$ of order $n$. Let $S$ denote the standard simplex given by 
\[S := \left\{\x = (x_v)_{v\in V(G)} \in \mathbb{R}^n : \sum_{v \in V(G)} x_v = 1,\ x_v \geq 0,\ v\in V(G)\right\}.\]
For a vector $\x\in \mathbb{R}^n$, the \emph{support} of $\x$ is given by $\supp(\x) = \{v\in V(G) : x_v\neq 0\}$.

Motzkin-Straus \cite{Motzkin_Straus_1965} established the following remarkable result concerning the clique number, which has found many applications in spectral Tur\'{a}n-type results; see \cite{ Kannan_Kumar_Pragada_2025, Kannan_Kumar_Pragada_2026,Liu_Ning_2025_weighted, Liu_Ning_2026}.

\begin{theorem}[\cite{Motzkin_Straus_1965}]\label{thm:Motzkin_Straus}
Let $G$  be a graph with adjacency matrix $A(G)$ and clique number $\omega(G)$. For any $\x \in S$, we have  
\begin{equation}\label{eq:Motzkin_Straus}
    \x^\top A(G)\x \leq 1- \frac{1}{\omega(G)}.
\end{equation}
\end{theorem}

Taking the complement of $G$, it is clear that \eqref{eq:Motzkin_Straus} is equivalent to the following inequality involving the independence number: for $\x\in S$, we have 
\begin{equation}\label{eq:Motzkin_Straus_independence}
    \x^\top (I + A(G))\x \ge \frac{1}{\alpha(G)}. 
\end{equation}

We first prove a new inequality similar in spirit to \eqref{eq:Motzkin_Straus_independence}, which we believe is of independent interest, and that will be key for us to establish the stronger Conjecture \ref{conj:Caro_Wei_localized} of Brause, Randerath, Rautenbach and Schiermeyer in Theorem \ref{thm:Caro_Wei_localized}. 

\begin{theorem}\label{thm:localized_MS_independence} Let $G$ be a graph with adjacency matrix $A(G)$. Consider the diagonal matrix $C(G):=\diag(c(v))_{v\in V(G)}$. Then, for any $\x\in S$, we have 
    \[  \x^\top (I + C(G) + A(G))\ \x \ge \frac{2}{\alpha(G)}. \]
Let $H_1, \ldots, H_k$ $(k\ge 1)$ denote the components of the induced subgraph $G[\supp(\x)]$. We have
\[\x^\top (I + C(G) + A(G))\ \x = \frac{2}{\alpha(G)}\]
if and only if all of the following hold:
\begin{enumerate}[$(i)$]
    \item $\alpha(G) = \alpha(G[\supp(\x)]) = \sum_{i=1}^k \alpha(H_i)$;
    \item $c_G(v) = c_{H_i}(v) = \omega(H_i)$ for all $v\in H_i$;
    \item for all $1\le i\le k$, there exists $d_i$ such that $H_i$ is $d_i$-regular and 
    \[\alpha(H_i)(d_i + \omega(H_i)+1) = 2|H_i|;\]
    \item for $v\in V(H_i)$, 
    \[ x_v = \frac{\alpha(H_i)}{\alpha(G)|H_i|}.\]
\end{enumerate}
\end{theorem}

\begin{proof} Throughout the proof, let
\[ Q(\x):= \x^\top (I + C(G) + A(G))\ \x.\] Let $\z\in S$ be a global minimizer of the quantity $Q(\x)$, i.e.,
\[ Q(\z) = \min_{\x\in S}Q(\x).\]
Such a $\z$ exists because $S$ is a compact set. For $v\in V(G)$, define
\[ f(v) := (c(v)+1)z_v + \sum_{u\in N(v)}z_u.\]
We now need to show two claims to complete the proof of the inequality.
\begin{claim}\label{claim:function_f}
For all $u\in \supp(\z)$, 
\[ f(u) = Q(\z).\]
\end{claim}

\begin{proof} 
Suppose, to the contrary, that there exists $u,v \in \supp(\z)$ such that $f(u) > f(v)$. Take $\y = \z + \varepsilon (\e_v - \e_u)$, where $\e_u\in \mathbb{R}^{V(G)}$ is the vector with $1$ at the entry corresponding to $u$ and $0$ elsewhere, and $\varepsilon$ is sufficiently small. It is clear that $\y \in S$. Then,
{\small{
\begin{align*}
Q(\y)
&=
\bigl(\z+\varepsilon(\e_v-\e_u)\bigr)^\top
(I + C(G) + A(G))
\bigl(\z+\varepsilon(\e_v-\e_u)\bigr)\\
&=
Q(\z)
+2\varepsilon(\e_v-\e_u)^\top (I + C(G) + A(G))\z
+\varepsilon^2(\e_v-\e_u)^\top
(I + C(G) + A(G))(\e_v-\e_u).
\end{align*}
}}
The linear term is
{\small{
\begin{align*}
(\e_v-\e_u)^\top (I + C(G) + A(G))\z
&=((I + C(G) + A(G))\z)_v-((I + C(G) + A(G))\z)_u \\
&=f(v)-f(u).
\end{align*}
}}
Furthermore,
\begin{align*}
(\e_v-\e_u)^\top (I + C(G) + A(G))(\e_v-\e_u)
&=(c(v)+1)+(c(u)+1)-2A_{uv}.
\end{align*}
We therefore have
\begin{align*}
Q(\y)-Q(\z) 
&=2\varepsilon(f(v)-f(u))+\varepsilon^2\left(c(v)+c(u)+2-2A_{uv}\right) \\
&=\varepsilon\bigl(2(f(v)-f(u))+\varepsilon (c(v)+c(u)+2-2A_{uv})\bigr).
\end{align*}
Notice that $c(v)+c(u)+2-2A_{uv}>0$. Hence, choosing $\varepsilon$ to be small enough, we obtain
\[
Q(\y)-Q(\z)<0.
\]
This contradicts the fact that $\z$ is a global minimizer of $Q$ over
$S$. Consequently, for any $u,v\in\supp(\z)$, we have $f(u)=f(v)$.  

Moreover, for any vertex $u\in \supp(\z)$,
\[ Q(\z) = \sum_{v\in V(G)}z_v f(v) = f(u)\sum_{v\in V(G)}z_v = f(u), \]
and the desired claim holds.
\end{proof}

To finish the proof of the assertion, it suffices to show the following claim.

\begin{claim}\label{claim:2_alpha_inequality}
    We have 
    \[Q(\z)\ge \frac{2}{\alpha(G)}.\]
\end{claim}

\begin{proof} 
Among all maximum independent sets in the subgraph $G[\supp(\z)]$ induced by $\supp(\z)$ in $G$, let $U$ be a maximum independent set that has the maximum value of $\sum_{u\in U}z_u$. We call such an independent set $U$ a \emph{good independent set}. 

Note that any vertex $w\in \supp(\z)\backslash U$ has at least one neighbour in $U$ by the maximality of $U$. For each $u\in U$, define 
\[ X_u:=\{w\in \supp(\z)\backslash U: N(w)\cap U=\{u\}\} \qquad \text{and}\qquad X:=\cup_{u\in U} X_u.\]
We say that the elements of $X_u$ are \emph{private neighbours} of $u$. 

First, see that $X_u$ is a clique, for otherwise two non-adjacent vertices in $X_u$ could replace $u$ in $U$ to produce a larger independent set in $G[\supp(\z)]$. Now, since every vertex in $X_u$ is adjacent to $u$, we have 
\begin{equation}\label{eq:size_X_u}
   |X_u|\le c(u)-1. 
\end{equation}
Second, for every $w\in X_u$, the set $(U\backslash \{u\})\cup \{w\}$ is a maximum independent set in $G[\supp(\z)]$. The maximality of $\z(U)$ implies that $z_w\le z_u$. Using \eqref{eq:size_X_u},
\begin{equation}\label{eq:sum_over_X_u}
    \sum_{w\in X_u}z_w\le (c(u)-1)z_u.
\end{equation}

Now,  
\begin{align}\label{eq:total_sum}
    |U|\ Q(\z) & = \sum_{u\in U} (c(u)+1)z_u  +  \sum_{w\in \supp(\z)\backslash U} |N(w)\cap U|\ z_w\qquad (\text{by Claim 2.1})\nonumber\\
    & \ge \sum_{u\in U} (c(u)+1)z_u  + \sum_{u\in U} \sum_{w\in X_u}z_w + \sum_{w\in \supp(\z)\backslash (U\cup X)} 2z_w \nonumber\\
    & = 2 +  \sum_{u\in U} \left((c(u)-1)z_u - \sum_{w\in X_u}z_w\right)\nonumber\\
    & \ge 2 \qquad (\text{by }\eqref{eq:sum_over_X_u}). 
\end{align}
This completes the proof of the claim.
\end{proof}

Now, we will analyze the equality case. First we consider the special case when $G$ is connected and $\supp(\z) = V(G)$.

\begin{claim}\label{claim:connected_full_support} Suppose $G$ is connected, $\supp(\z) = V(G)$ and 
\[Q(\z) = \frac{2}{\alpha(G)}.\]
Then all of the following hold:
\begin{itemize}
    \item $c(v) = \omega(G)$; 
    \item $\z = \frac{1}{n}\1$;
    \item $G$ is $\Delta(G)$-regular and $\alpha(G) = \frac{2n}{(\Delta(G)+\omega(G)+1)}$.
\end{itemize}
\end{claim}

\begin{proof}
Since $Q(\z) = \frac{2}{\alpha(G)}$, all inequalities must be equalities in the proof of Claim \ref{claim:2_alpha_inequality}. 

By \eqref{eq:size_X_u}, \eqref{eq:sum_over_X_u} and \eqref{eq:total_sum}, for any $u\in U$, we have 
\begin{equation}\label{eq:private_neighbour}
   c(u) - 1 = |X_u|, \qquad z_u = z_w \quad (w\in X_u), 
\end{equation}
and every vertex $a\in V(G)\backslash (U\cup X)$ has exactly two neighbours in a good independent set $U$. Clearly, $X_u \cup \{u\}$ is a clique of order $c(u)$ which shows that $c(u)\le c(w)$ $(w\in X_u)$. 

For $w\in X_u$, the set 
\[U' = (U\backslash\{u\})\cup \{w\}\]
is also a good independent set by \eqref{eq:private_neighbour}. Relative to $U'$, $u$ is a private neighbour of $w$. Repeating the argument with $U'$ shows that $c(w)\le c(u)$. Thus we can conclude that 
\begin{equation}\label{eq:private_same_clique_size}
  z_u=z_w \quad \text{and}\quad c(u) = c(w)\quad (w\in X_u).  
\end{equation}

Now, let $a\in V(G)\backslash (U\cup X) = \supp(\z)\backslash (U\cup X)$. Then, 
\[N(a)\cap U = \{u,v\}\]
for distinct $u,v\in V(G)$. Since $X_u\cup \{u\}$ is a maximum clique containing $u$, $a$ cannot be adjacent to all vertices in $X_u$. So let $w\in X_u\backslash N(a)$. Relative to the good independent set
\[U' = (U\backslash\{u\})\cup \{w\},\]
$a$ is a private neighbour of $v$. Hence, using \eqref{eq:private_same_clique_size}, we have 
\[    z_a = z_v\quad \text{and}\quad  c(a) = c(v).\]
Interchanging the roles of $u$ and $v$, we get
\begin{equation}\label{eq:two_neighbour}
z_a =z_u = z_v\quad \text{and}\quad c(a)=c(u)=c(v). 
\end{equation}

Consider an edge $ab\in E(G)$. Using the above observations we will argue that $z_a =z_b$ and $c(a) = c(b)$. We consider the following cases:

\begin{description}
    \item[Case 1:] $a\in U$, $b\in X_a$. ~ Follows by \eqref{eq:private_same_clique_size}. 
    \item[Case 2:] $a\in U$, $b\in V(G)\backslash (U\cup X)$. ~ Follows by
\eqref{eq:two_neighbour}.
\item[Case 3:] $a, b\in X_u$ for some $u\in U$. ~ Follows by \eqref{eq:private_same_clique_size}.
\item[Case 4:] $a\in X_u$ and $b\in X_v$ $(u\neq v)$. ~ Then relative to good independent set $U'= (U\backslash\{u\})\cup \{a\}$, $b$ has two neighbours in $U'$ (namely, $v$ and $a$), and so by \eqref{eq:two_neighbour}, we have $z_b = z_a$ and $c(b)=c(a)$.

\item[Case 5:] $a\in X_u$ and $b\in V(G)\backslash (U\cup X)$. ~
 Then $ub\in E(G)$, for otherwise, $b$ has three neighbours in $U'= (U\backslash\{u\})\cup \{a\}$, which is impossible. By \eqref{eq:private_same_clique_size} and \eqref{eq:two_neighbour}, we have $z_a = z_u = z_b$ and $c(a)=c(u)=c(b)$. 
 \item[Case 6:] $a,b\in V(G)\backslash (U\cup X)$. ~ If $N(a)\cap N(b)\cap U\neq \emptyset$, then we are done by \eqref{eq:two_neighbour}. So assume that $(N(a)\cap U) = \{u,v\}$ and $(N(b)\cap U)=\{p,q\}$ are disjoint. We can find $w\in X_u$ such that $w$ is non-adjacent to $a$. Then 
\[U'=(U\backslash\{u,v\})\cup \{w,a\}\]
is a good independent set. But $b$ has three neighbours in $U'$ (namely, $a,p,q$), which is impossible. 
\end{description}

We have shown that for any edge $ab\in E(G)$, we have $z_a = z_b$ and $c(a)=c(b)$. Since $G$ is connected, we conclude that $c(v) = \omega(G)$ and $z_v = \frac{1}{n}$ for all $v\in V(G)$.

Finally, by Claim \ref{claim:function_f}, we see that 
\[ \deg(v) = \frac{2n}{\alpha(G)} - \omega(G)-1\]
for all $v\in \supp(\z) = V(G)$. Thus, $G$ is $\Delta(G)$-regular and $\alpha(G) = \frac{2n}{\Delta(G)+\omega(G)+1}$. The proof of the claim is complete. 
\end{proof}

Next, we analyze the equality in the general case. 

\begin{claim} Let $\x\in S$ and let $H_1, \ldots, H_k$ $(k\ge 1)$ denote the components of the induced subgraph $G[\supp(\x)]$. We have
\[Q(\x) = \frac{2}{\alpha(G)}\]
if and only if all of the following hold:
\begin{enumerate}[$(i)$]
    \item $\alpha(G) = \alpha(G[\supp(\x)]) = \sum_{i=1}^k \alpha(H_i)$;
    \item $c_G(v) = c_{H_i}(v) = \omega(H_i)$ for all $v\in H_i$;
    \item for all $1\le i\le k$, there exists $d_i$ such that $H_i$ is $d_i$-regular and 
    \[\alpha(H_i)(d_i + \omega(H_i)+1) = 2|H_i|;\]
    \item for $v\in V(H_i)$, 
    \[ z_v = \frac{\alpha(H_i)}{\alpha(G)|H_i|}.\]
\end{enumerate}
\end{claim}

\begin{proof} Let $H$ denote a subgraph of $G$ and let $\x\in \mathbb{R}^{V(H)}$ be such that $x_v\ge 0$ and $\sum_{v\in V(H)}x_v = 1$. Define 
\[Q_H(\x) = \x^\top (I+C(H)+A(H))\x,\] 
where $C(H) = \diag(c_H(v))_{v\in V(H)}$ and $A(H)$ is the adjacency matrix of $H$. 

We first prove necessity. By assumption  $\x\in S$ is a global minimizer of $Q(\x)$. For $1\le i\le k$, define
\[\x^{(i)}:=\frac{\x|_{H_i}}{\sum_{v\in V(H_i)}x_v},\]
where $\x|_{H_i}$ denotes the restriction of $\x$ to $H_i$. Then $Q_{H_i}(\x^{(i)})$ is well-defined.

Now, observe that 
\begin{align*}
    Q(\x) & = \sum_{i=1}^k \left(\sum_{v\in V(H_i)}x_v\right)^2\left[Q_{H_i}(\x^{(i)}) + \sum_{v\in V(H_i)}(c_G(v)-c_{H_i}(v))(\x^{(i)}_v)^2\right]\\
    & \ge \sum_{i=1}^k \left(\sum_{v\in V(H_i)}x_v\right)^2 \, Q_{H_i}(\x^{(i)})\quad (\text{since }c_G(v)\ge c_{H_i}(v))\\
    & \ge \sum_{i=1}^k \left(\sum_{v\in V(H_i)}x_v\right)^2 \frac{2}{\alpha(H_i)}\quad (\text{Claim \ref{claim:2_alpha_inequality} applied to $H_i$})\\
    & \ge 2\frac{\left(\sum_{i=1}^k \left(\sum_{v\in V(H_i)}x_v\right)\right)^2}{\sum_{i=1}^k \alpha(H_i)}\quad (\text{Cauchy-Schwarz})\\
    & = \frac{2}{\sum_{i=1}^k \alpha(H_i)}\\
    & \ge \frac{2}{\alpha(G)}.
\end{align*}
Since by assumption, $Q(\x) = \frac{2}{\alpha(G)}$, it follows that all of the inequalities above must be equalities. Indeed, we must have the following:
\begin{enumerate}[$(a)$]
    \item $\alpha(G) = \sum_{i=1}^k \alpha(H_i)$ which implies $(i)$;
    \item $c_G(v)=c_{H_i}(v)$ for all $v\in V(H_i)$;
    \item $Q_{H_i}(\x^{(i)}) = \frac{2}{\alpha(H_i)}$;
    \item $\frac{\sum_{v\in V(H_i)}x_v}{\alpha(H_i)}$ is the same for all $i$, which implies 
    \[\sum_{v\in V(H_i)} x_v= \frac{\alpha(H_i)}{\alpha(G)}.\] 
\end{enumerate}

By $(c)$, we conclude that $\x^{(i)}$ is a global minimizer of $Q_{H_i}(\x)$. Applying Claim \ref{claim:connected_full_support} to $H_i$ and $\x^{(i)}$, the following hold:
\begin{itemize}
    \item $c_{H_i}(v) = \omega(H_i)$ for all $v\in H_i$, which implies $(ii)$ by $(b)$;
    \item assertion $(iii)$ for all $i$;
    \item and $\x^{(i)} = \frac{1}{|H_i|}\1$, which implies $(iv)$  by $(d)$. 
\end{itemize}

Now, we discuss the sufficiency. If $\x$ and $G$ satisfy the four conditions $(i)-(iv)$, then the contribution of $H_i$ to $Q(\x)$ is $\frac{2\alpha(H_i)}{\alpha(G)^2}$. Summing over $i$ gives $Q(\x) = \frac{2}{\alpha(G)}$.  
\end{proof}
The proof of Theorem \ref{thm:localized_MS_independence} is complete.
\end{proof}

\begin{remark}
It is not true that $\x^\top (C(G)+I)\x \ge \frac{1}{\alpha(G)}$ for any $\x\in S$. For instance, if $G$ is the Clebsch graph, it is triangle-free with $n=16$, $\alpha(G) = 5$ and $c(v)=2$. For the normalized all ones vector $\x = \frac{1}{16}\1$, we have
\[\x^T(C(G)+I)\x = \frac{3}{16} < \frac{1}{5} = \frac{1}{\alpha(G)}.\]
\end{remark}

\section{A proof of the localized Caro--Wei bound}
\label{section:Caro_Wei_localization}

The following result establishes the stronger Conjecture \ref{conj:Caro_Wei_localized}, which also implies Conjecture \ref{conj:Bertram_Horak}.

\begin{theorem}\label{thm:Caro_Wei_localized}
For any graph $G$,  
  \[\alpha(G)\geq
  \sum_{v\in V(G)}\frac{2}{\deg(v)+c(v)+1}.\]
Moreover, equality holds if and only if $G$ is the disjoint union of $H_1, \ldots, H_k$ $(k\ge 1)$ such that $H_i$ is a $\Delta(H_i)$-regular graph and attains Fajtlowicz’s bound, i.e.,
\[ \alpha(H_i) = \frac{2|H_i|}{\Delta(H_i)+\omega(H_i)+1}\quad (1\le i\le k).\]
\end{theorem}

\begin{proof}
Consider the following vector $\x \in \mathbb{R}^{V(G)}$, given by
\[
  x_v:=\frac{1}{\deg(v) + c(v)+ 1}. 
\]
Substituting this particular choice of vector $\x /\|\x\|_1$ into Theorem \ref{thm:localized_MS_independence}, we have
\[\x^T (I + C(G) + A(G)) \x \geq \frac{2\|x\|_1^2}{\alpha(G)}.\]
Now, for any edge $uv\in E(G)$, the AM-GM inequality gives $2x_ux_v\leq x_u^2+x_v^2$. Therefore
\begin{align*}
  \x^T (I + C(G) + A(G)) \x
  &\leq \sum_{v\in V(G)}(c(v)+1)x_v^2
  +\sum_{uv\in E(G)}(x_u^2+x_v^2)\\
  &=\sum_{v\in V(G)}(\deg(v)+c(v)+1)x_v^2\\
  &=\sum_{v\in V(G)} x_v = \|\x\|_1.
\end{align*}
Combining the above inequalities gives
\[
  \frac{2\|\x\|_1^2}{\alpha(G)}\leq \|\x\|_1,
\]
which implies
\[
  \alpha(G)\geq2\|\x\|_1
  =\sum_{v\in V(G)}\frac{2}{\deg(v)+c(v)+1},
\]
as desired. Moreover, the characterization for the equality follows from the equality case in Theorem \ref{thm:localized_MS_independence} applied to  $\frac{1}{\|x\|_1}\x$ and $G$. 
\end{proof}

\begin{remark}\label{remark:equality_Fajtlowicz}
Fajtlowicz \cite{Fajtlowicz_1984} investigated the tightness in Theorem \ref{thm:Fajtlowicz}. In particular, it is known that Fajtlowicz's bound is tight for complete graphs, the $5$-cycle, the complement of the Clebsch graph, and other graph classes. Further examples arise from triangle-free regular graphs whose independence number is equal to their degree. Sidorenko \cite{Sidorenko_1991} constructed a rich family of such graphs, and these graphs, together with further constructions, were subsequently studied by Brandt \cite{Brandt_2010}. The complements of all such graphs are tight for the Fajtlowicz bound. Moreover, if $G$ is tight for the Fajtlowicz bound, then its \emph{closed blowup} $G^{[t]}$ $(t\ge 1)$ is also tight (recall that $G^{[t]}$ is the graph obtained from $G$ by replacing every vertex $v\in V(G)$ with a clique $C_v$ of size $t$ and joining $C_u$ and $C_v$ whenever $uv\in E(G)$).
\end{remark}

Theorem \ref{thm:Caro_Wei_localized} has consequences beyond the study of the independence number; in particular,  
it yields a new Tur\'an-type inequality involving both the clique number and the independence number. In the next section, we will see that through the graph blow-up method, this inequality can be translated into a Motzkin--Straus type theorem, which we will then use to investigate bipartiteness in $K_{r+1}$-free graphs in Sections \ref{section:spectral_bipartiteness} and \ref{section:algebraic_bipartiteness}.

\section{The graph blow-up trick}
\label{section:graph_blowup_trick}
The classical Tur\'{a}n's inequality states that a graph $G$ of order $n$ and size $m$ satisfies:
\begin{equation}\label{eq:Turan}
    \frac{2m}{n^2}\le 1 - \frac{1}{\omega(G)}.
\end{equation}

The inequality \eqref{eq:Turan} and the Motzkin-Straus Theorem \ref{thm:Motzkin_Straus} are indeed equivalent statements, as pointed out by Nikiforov \cite{Nikiforov_2022} and formally proved by Sidorenko \cite{Sidorenko_1987}. This equivalence is typically established using the graph blow-up trick; see Phan \cite{Phan_2026} for a recent exposition.

We establish a new inequality which is similar but not comparable to \eqref{eq:Turan}. It follows as a consequence of Theorem \ref{thm:Caro_Wei_localized}. 

\begin{theorem}[Clique-independence edge inequality]\label{thm:clique_independence_edge} For any graph $G$ of order $n$ and size $m$, we have 
\begin{equation}\label{eq:clique_independence_edge}
    \frac{2m}{n^2}\le 1 - \frac{2}{\omega(G)} + \frac{\alpha(G)}{n}.
\end{equation}
\end{theorem}

\begin{proof} We apply Theorem \ref{thm:Caro_Wei_localized} to the complement $\overline{G}$ of $G$. We have 
\begin{align*} 
\omega(G) & = \alpha(\overline{G})\\
& \ge \sum_{v\in V(\overline{G})} \frac{2}{\deg_{\overline{G}}(v) + \omega(\overline{G}) + 1}\\
& \ge  \frac{2n^2}{\left(\sum_{v\in V(\overline{G})} \deg_{\overline{G}}(v)\right) + n(\omega(\overline{G}) + 1)} \qquad (\text{by Titu's Lemma})\\
& = \frac{2n^2}{(n(n-1)-2m) + n(\alpha(G)+1)}\\
& = \frac{2n^2}{n(n+\alpha(G))-2m}.
\end{align*}
Rearranging gives the desired inequality \eqref{eq:clique_independence_edge}. 
\end{proof}

\begin{remark}
   The inequality \eqref{eq:clique_independence_edge} is tight for $C_5$, whereas classical Tur\'{a}n's inequality \eqref{eq:Turan} is not tight. For diamond $K_4-e$, classical Tur\'{a}n's inequality \eqref{eq:Turan} gives a better upper bound than the inequality \eqref{eq:clique_independence_edge}. Hence, these two inequalities are not comparable.
\end{remark}

Now, using the blow-up trick, one can rewrite \eqref{eq:clique_independence_edge} as an inequality analogous to the Motzkin-Straus Theorem, as we show below.

\begin{theorem}\label{thm:weighted_clique_indepenence} Let $G$ be a graph with $\omega(G)\ge 2$. For any $\x\in S$, we have 
\begin{equation}\label{eq:weighted_clique_indepenence}
    \x^\top A(G)\x \le 1-\frac{2}{\omega(G)} + \alpha_{\x}(G),
\end{equation}
where 
\[ \alpha_{\x}(G) := \max\left\{\sum_{v\in U} x_v : U \text{ is an independent set in }G\right\}.\]
\end{theorem}

\begin{proof} We will first prove the assertion for entrywise positive rational vectors $\x$. Choose a positive integer $N$ large enough so that $Nx_v$ is an integer for all $v\in V(G)$. Let $H$ denote the graph obtained from $G$ by replacing vertex $v$ with an independent set $U_v$ of size $Nx_v$ and joining $U_v$ and $U_w$ whenever $vw\in E(G)$. The graph $H$ is called a \emph{blowup} of $G$. Clearly, 
\begin{equation}\label{eq:blowup_parameters}
    \sum_{v\in V(G)} Nx_v = N, \quad \omega(H) = \omega(G), \quad \alpha(H) = N \alpha_{\x}(G), \quad \frac{2|E(H)|}{N^2} = \x^{\top}A(G)\x.
\end{equation}
Thus, by Theorem \ref{thm:clique_independence_edge}, we have 
\begin{align*}
    \frac{2|E(H)|}{N^2}\le 1-\frac{2}{\omega(H)} + \frac{\alpha(H)}{N}.
\end{align*}
Using \eqref{eq:blowup_parameters}, we see that the above inequality is equivalent to the desired inequality. This completes the proof for entrywise positive rational vectors in $S$. 

Since entrywise positive rational vectors are dense in $S$, a limiting argument proves the inequality for all vectors in $S$ since the left-hand side and the right-hand side of \eqref{eq:weighted_clique_indepenence} are continuous functions of $\x$. 
\end{proof}

\begin{remark}
The proof of Theorem \ref{thm:weighted_clique_indepenence} shows that it is implied by Theorem \ref{thm:clique_independence_edge}. Conversely, taking $\x = \frac{1}{n}\1$, where $\1$ is the all-ones vector in Theorem \ref{thm:weighted_clique_indepenence} recovers Theorem \ref{thm:clique_independence_edge}. 
\end{remark}

\section{Spectral bipartiteness in $K_{r+1}$-free graphs}
\label{section:spectral_bipartiteness}

In this section, we analyse the spectral bipartiteness of $K_{r+1}$-free graphs, and we extend Brandt's bound from Theorem \ref{thm:spectral_bipartiteness_r_old} $(i)$  to all $K_{r+1}$-free graphs. The main goal is to prove Theorems \ref{thm:q_min_new_bound} and \ref{thm:K4_free_counterexample}.

First, we recall a Perron-weighted generalization of the well-known Hoffman ratio bound proved by Fiol \cite{Fiol_1999} that we will use. Note that the entries of the Perron vector quantify the relative structural importance of individual vertices, concentrating heavily on vertices with high local degrees or those belonging to dense subgraphs (cliques). Since then, weight partitions using Perron eigenvector entries have been shown to be very useful to tackle graph theory problems, see e.g. \cite{Abiad_2019_weight_regular,Liu_Ning_2026}.
First, we recall a Perron-weighted generalization of the well-known Hoffman ratio bound proved by Fiol \cite{Fiol_1999} that we will use. Note that the entries of the Perron vector quantify the relative structural importance of individual vertices, concentrating heavily on vertices with high local degrees or those belonging to dense subgraphs (cliques). Since then, weight partitions using Perron eigenvector entries have been shown to be very useful to tackle graph theory problems, see e.g. \cite{Abiad_2019_weight_regular,Fiol_1999,Liu_Ning_2026}.

\begin{theorem}[Perron-weighted Hoffman ratio bound \cite{Fiol_1999}]\label{thm:weighted_Hoffman} Let $G$ be a graph and $\x$ be a non-negative unit eigenvector for $\lambda_1(G)>0$. Let $U$ be an independent set. Then 
\[ \sum_{u\in U}x_u^2 \le \frac{-\lambda_n(G)}{\lambda_1(G)-\lambda_n(G)}.\]
\end{theorem}

We also recall here a well-known result of Wilf \cite{Wilf_1986}. 

\begin{theorem}[Wilf's inequality \cite{Wilf_1986}]\label{thm:wilf_inequality}
    For any graph $G$ of order $n$, we have
    \[\lambda_1(G) \leq n\left(1 - \frac{1}{\omega(G)}\right).\]
\end{theorem}

We are now ready to prove the main result of this section. 

\begin{theorem}
For any graph $G$ of order $n$, we have
 \[\lambda_1(G) + \lambda_n(G) \leq n \left(2\sqrt{1-\frac{1}{\omega(G)}}-1\right)^2.\]
\end{theorem}

\begin{proof}
We can assume that $G$ is non-empty. Let $\x$ denote the unit non-negative eigenvector corresponding to $\lambda_1$. Let $\y = (y_v)_{v\in V(G)}$ be the vector such that $y_v = x_v^2$ for all $v\in V(G)$. Note that 
\[ \lambda_1^2 = \left(2\sum_{uv\in E(G)}x_ux_v\right)^2 \le 4m\sum_{uv\in E(G)}x_u^2x_v^2 = 2m\, \y^\top A(G)\y.\]

Using Theorem \ref{thm:weighted_clique_indepenence}, 
\begin{align}\label{eq:alpha_y}
  \alpha_{\y}(G)& \ge \y^\top A(G)\y - \left(1- \frac{2}{\omega}\right)\nonumber\\
  & \ge \frac{\lambda_1^2}{2m} - \left(1- \frac{2}{\omega}\right) \nonumber\\
  & \ge \frac{\lambda_1}{n} - \left(1- \frac{2}{\omega}\right),
\end{align}
where the last inequality holds since $\lambda_1(G)\ge \frac{2m}{n}$. We consider the following cases:

\textbf{Case 1:} $\frac{\lambda_1}{n} \le 1- \frac{2}{\omega}$.

Since $G$ is a non-empty graph, $\lambda_n<0$. Thus, 
\[ \frac{\lambda_1 + \lambda_n}{n}\le \frac{\lambda_1}{n}\le 1-\frac{2}{\omega}\le \left(2\sqrt{1-\frac{1}{\omega}}-1\right)^2.\]

\textbf{Case 2:}  $\frac{\lambda_1}{n} > 1- \frac{2}{\omega}$.

Note that the function $x \mapsto \frac{x}{1-x}$ is an increasing function on $[0,1)$. Hence, using Theorem \ref{thm:weighted_Hoffman} and \eqref{eq:alpha_y},
\[-\lambda_n \ge \frac{\lambda_1 \alpha_{\y}}{1-\alpha_{\y}}\ge  \frac{\lambda_1\left( \frac{\lambda_1}{n} - \left(1- \frac{2}{\omega}\right)\right)}{1 - \left( \frac{\lambda_1}{n} - \left(1- \frac{2}{\omega}\right)\right)}=\frac{\lambda_1^2 - n\lambda_1\left(1- \frac{2}{\omega}\right)}{n-\lambda_1+n\left(1- \frac{2}{\omega}\right)}.\]
Therefore, 
\[\lambda_1+\lambda_n \le \lambda_1 - \frac{\lambda_1^2 - n\lambda_1\left(1- \frac{2}{\omega}\right)}{n-\lambda_1+n\left(1- \frac{2}{\omega}\right)}.\]
Treating the right-hand side as a function of $\lambda_1$ and assuming that $n$ and $\omega \ge 1$ are fixed, one can see that the maximum value of the right-hand side over the interval $\lambda_1 \in \left[0, n(1-\frac{1}{\omega})\right]$ (by Theorem \ref{thm:wilf_inequality}) is 
\[n \left(2\sqrt{1-\frac{1}{\omega}}-1\right)^2\]
attained at $\lambda_1 = 2n-\frac{2n}{\omega} - n\sqrt{1-\frac{1}{\omega}}$. The assertion follows.
\end{proof}

\section{Algebraic bipartiteness in $K_{r+1}$-free graphs}
\label{section:algebraic_bipartiteness}

We recall a generalization of Motzkin-Straus and spectral Tur\'{a}n's theorem to doubly non-negative matrices given by Liu, Tang, and Zhang \cite{Liu_Tang_Zhang_2026}. A matrix $M$ is called \emph{doubly non-negative} if it is both positive semi-definite and entrywise non-negative. 

\begin{theorem}[\cite{Liu_Tang_Zhang_2026}]\label{thm:DNN_Motzkin_Straus}
For any graph $G$ of order $n$, and every doubly
non-negative matrix $M$, we have
\[\sum_{uv \in E(G)}2\sqrt{M_{uv}} \leq n \left(1- \frac{1}{\omega(G)}\right)\sqrt{\sum_{u,v \in V(G)} M_{uv}} .\]
\end{theorem}

Now, we are ready to prove Theorem \ref{thm:q_min_new_bound}, which improves the upper bound on algebraic bipartiteness of de Lima et al. (see Theorem \ref{thm:spectral_bipartiteness_r_old} $(ii)$). We first prove the following more general result, which bounds the algebraic bipartiteness of a graph in terms of its clique number. Theorem \ref{thm:q_min_new_bound} follows immediately by applying the following result to $K_{r+1}$-free graphs.

\begin{theorem}
    For any graph $G$ of order $n$, we have
    \[\frac{q_n(G)}{n} \leq \left(1-\frac{1}{\omega(G)}\right)^2.\]
\end{theorem}

\begin{proof}
    If the graph $G$ is empty, then it is clear that the result holds, hence we can assume without loss of generality, that the graph $G$ is non-empty. Since $Q$ is a positive semi-definite matrix, it is easy to see that $q_n$ is less than the minimum degree of the graph $G$, and thus 
    \[M = Q - q_nI,\]
    is a positive semi-definite and non-negative matrix, hence $M$ is a doubly non-negative matrix.

    Applying Theorem \ref{thm:DNN_Motzkin_Straus} to $M$ and noting that $M_{uv} =1$, whenever $uv \in E(G)$, we get
    \begin{align*}
        2m \leq n  \left(1- \frac{1}{\omega}\right) \sqrt{4m - nq_n}.
    \end{align*}
    Squaring and rearranging, we get
    \[q_n \leq \frac{4m}{n} - \frac{(2m)^2}{n^3\left(1- \frac{1}{\omega}\right)^2}.\]
    Dividing further by $n$, we get
    \begin{align*}
        \frac{q_n}{n} &\leq \frac{4m}{n^2}  - \frac{(2m)^2}{n^4 \left(1- \frac{1}{\omega}\right)^2} \\
        & = \left(1- \frac{1}{\omega}\right)^2 - \frac{\left(\frac{2m}{n^2} - \left(1- \frac{1}{\omega}\right)^2\right)^2}{\left(1- \frac{1}{\omega}\right)^2} \\
        & \leq \left(1- \frac{1}{\omega}\right)^2. 
    \end{align*}
This completes the proof.     
\end{proof}

Next, we will show Theorem \ref{thm:K4_free_counterexample}, thus disproving Conjecture \ref{conj:algebraic_K_r_free} for $r=3$. To do so we will need the following result on the signless Laplacian spectra of join of two regular graphs.

\begin{theorem}[\cite{Barik_Kalita_Pati_Sahoo_2018}]\label{thm:q_join_spectrum}
Let $G_1$ be an $r_1$-regular graph on $n_1$ vertices and $G_2$ be an $r_2$-regular graph on $n_2$ vertices. Then the signless Laplacian spectrum of $G_1\lor G_2$ is given by
\[q_i(G_1) + n_2,\qquad q_j(G_2)+n_1,\]
and 
\[\frac{n_1+n_2 + 2(r_1+r_2) \pm \sqrt{(n_1+n_2)^2-4(r_1-r_2)(n_1-n_2-r_1+r_2)}}{2},\]
where $2\le i \le n_1$ and $2\le j\le n_2$.
\end{theorem}

Now, we are ready to prove Theorem \ref{thm:K4_free_counterexample}.
\begin{proof}[Proof of Theorem \ref{thm:K4_free_counterexample}]

For integers $t\geq 1$ and $n-5t\geq 2$, let
\[
G(t):=\overline{K}_{n-5t}\vee C_5^{(t)},
\]
where $C_5^{(t)}$ denotes the graph obtained from the cycle $C_5$ by replacing every vertex $v$ with an independent set $U_v$ of size $t$ and add all edges between $U_u$ and $U_v$ whenever $u$ is adjacent to $v.$
The graph $C_5^{(t)}$ is triangle-free, and so $\omega(G(t)) = 3$. 

The eigenvalues of $Q(C_5^{(t)})$ are
\[
4t^{(1)},
\frac{(3+\sqrt5)t}{2}^{(2)},
2t^{(5t-5)}, 
\frac{(3-\sqrt5)t}{2}^{(2)}.
\]

Applying Theorem~\ref{thm:q_join_spectrum}, we have
\begin{align} \label{eq:qmin_values}
q_n(G(t))
=
\min\left\{
5t,\,
n-\frac{7+\sqrt5}{2}t,\,
\frac{n+4t-\sqrt{n^2+8nt-64t^2}}{2}
\right\}.  
\end{align}

For all sufficiently large $n$, set
\[
t_n
:=
\left\lfloor
\frac{15+\sqrt5}{95+11\sqrt5}\,n
\right\rfloor.
\]
Then, we have $t_n\geq 1$ and $n-5t_n\geq 2$ for all sufficiently large
$n$, and so  
\[
G_n
:= G(t_n)
\]
is well-defined. Using \eqref{eq:qmin_values} observe that
\[\lim_{n\to \infty}\frac{q_n(G_n)}{n} = \min\left\{5\frac{15+\sqrt5}{95+11\sqrt5},\, \frac{60-4\sqrt5}{137+7\sqrt5},\, \frac{60-4\sqrt5}{137+7\sqrt5}\right\} = \frac{60-4\sqrt5}{137+7\sqrt5} > \frac{1}{3}.\]
This completes the proof. 
\end{proof}

\section{Open problems}
\label{section:open_problems}
In this paper, we improved the general upper bounds for spectral and algebraic bipartiteness of $K_{r+1}$-free graphs and, along the way, also proved Conjecture \ref{conj:Caro_Wei_localized} and disproved Conjecture \ref{conj:algebraic_K_r_free} when $r = 3$. Several questions remain open. 

For instance, the triangle-free graph case still remains a challenge.

\begin{conjecture} Let $G$ be a triangle-free graph of order $n$. 
\begin{enumerate}[$(i)$]
    \item \emph{(cf. \cite{Balogh_Clemen_Felix_Lidicky_Norin_Volec_2023})} We have $q_n(G) \leq 0.14 n$.
    \item We have $\lambda_1(G) + \lambda_n(G) \leq 0.14 n$.
\end{enumerate}
\end{conjecture}

Note that the bound $0.14n$ is attained by the Higman--Sims graph, which is a strongly regular graph $\SRG(100,22,0,6)$, and its blowups for both $q_n$ and $\lambda_1 + \lambda_n$.  

For higher $K_{r+1}$-free graphs, we pose the following conjecture. 

\begin{conjecture} Let $G$ be a $K_{r+1}$-free of order $n$.
\begin{enumerate}[$(i)$]
    \item For $r \ge 4$, we have
     \[ \frac{q_n(G)}{n} \leq  1 - \frac{2}{r}.\]
     \item For $r \ge 3$, we have 
    \[ \frac{\lambda_1(G)+\lambda_n(G)}{n} \leq  1 - \frac{2}{r}.\]
\end{enumerate}
\end{conjecture}

It is interesting to note that our counterexample for $q_n$ when $r=3$ (Theorem \ref{thm:K4_free_counterexample}) does not work for $\lambda_1+\lambda_n$. 

We note that spectral bipartiteness $\lambda_1 + \lambda_n$ has also been investigated for graphs with high odd girth by Abiad, Taranchuk, and Veluw \cite{Abiad_Taranchuk_Veluw_2026} and Yip \cite{Yip_2026}. A similar question can be asked for the algebraic bipartiteness $q_n$. 

Finally, characterizing the equality for the Fajtlowicz bound is an interesting and challenging problem; see Remark \ref{remark:equality_Fajtlowicz}.

\subsection*{Acknowledgements}
Aida Abiad is supported by NWO (Dutch Research Council) through the grant \linebreak  VI.Vidi.213.085. This work was done in part while the first author was visiting the Simons Institute for the Theory of Computing at UC Berkeley.

\subsection*{Declaration of AI use}

The authors acknowledge the use of ChatGPT (GPT-5.6, OpenAI; accessed August 2026) solely for preliminary brainstorming and the exploration of possible proof strategies. AI tools were not used to draft the manuscript. All formal statements, arguments, and proofs in the manuscript were written and checked by the authors, who take full responsibility for the accuracy and integrity of the article.

\bibliographystyle{abbrv}
\bibliography{references}

\vspace{0.4cm}

\affl{Aida Abiad}{a.abiad.monge@tue.nl}{Department of Mathematics and Computer Science, Eindhoven University of Technology, The Netherlands \\ Department of Mathematics and Data Science, Vrije Universiteit Brussel, Belgium}

\affl{Hitesh Kumar}{hitesh.kumar.math@gmail.com, hitesh\_kumar@sfu.ca}{Department of Mathematics, Simon Fraser University, Burnaby, Canada}

\affl{Shivaramakrishna Pragada}{shivaramkratos@gmail.com, shivaramakrishna\_pragada@sfu.ca}{Department of Mathematics, Simon Fraser University, Burnaby, Canada}

\end{document}